\documentclass[a4paper, 11pt]{article}
\usepackage{mathrsfs}
\usepackage{amsfonts}
\usepackage{latexsym}
\usepackage{graphicx}
\usepackage{amsmath}
\newcommand{\new}{\newcommand*}

\new{\rnew}{\renewcommand*}
\new{\newe}{\newenvironment*}
\new{\newt}{\newtheorem}
\new{\stl}{\setlength}

\new{\bea}{\begin{eqnarray}}
\new{\eea}{\end{eqnarray}}

\new{\be}{\begin{equation}}
\new{\ee}{\end{equation}}

\new{\bean}{\begin{eqnarray*}}
\new{\eean}{\end{eqnarray*}}

\new{\no}{\nonumber}

\new{\bt}{\begin{theorem}}
\new{\et}{\end{theorem}}
\new{\bl}{\begin{lemma}}
\new{\el}{\end{lemma}}

\new{\bc}{\begin{corollary}}
\new{\ec}{\end{corollary}}
\new{\bp}{\begin{proof}\quad}
\new{\ep}{\end{proof}}
\new{\ba}{\begin{array}}
\new{\ea}{\end{array}}

\newt{proposition}{Proposition}[section]
\newt{lemma}{Lemma}[section]
\newt{conjecture}{Conjecture}[section]
\newt{corollary}{Corollary}[section]
\newt{theorem}{Theorem}[section]
\newt{assumption}{Assumption}
\newt{example}{Example}[section]
\newt{remark}{Remark}[section]
\newt{definition}{definition}[section]

\rnew{\theequation}{\thesection.\arabic{equation}}
\new{\sect}[1]{ \section{#1}
\setcounter{equation}{0} \setcounter{figure}{0} }

\newe{acknowledgment}{{\flushleft \bf Acknowledgments. }}{}
\newe{proof}{{\bf Proof.}}{\qquad\endproof}
\newe{keywords}{\begin{quote}\quad{\bf Keywords.}}{\end{quote}}
\newe{AMS}{\begin{quote}\quad {\bf AMS subject classification.}}{\end{quote}}

\def \endproof {\qquad \vrule height 5pt width 5pt depth 2pt}

\title{Carleson measures on planar sets\date{}}

\author{{Zhijian Qiu     \thanks{ }} \\
\\{\small\em  Department of  mathematics,
Southwestern University of Finance and Economics}\\
}

\begin{document}
\maketitle
\begin{abstract} In this paper, we investigate what are  Carleson measures on open subsets in
the complex plane.  A circular domain is a connected open subset
whose boundary consists of finitely many disjoint circles. We call a
domain $G$ multi-nicely connected if there exists a circular domain
$W$ and a conformal map $\psi$ from $W$ onto $G$ such that $\psi$ is
almost univalent with respect the arclength on $\partial W$.  We
characterize all Carleson measures for those open subsets so that
each of their components is multi-nicely connected and harmonic
measures of the components are mutually singular.  Our results
suggest the extend of Carleson measures probably is up to this class
of  open subsets.
\end{abstract}

\begin{keywords} Carleson measure, circular domain, harmonic measure, multi-nicely
connected domain
\end{keywords}

\begin{AMS}    30H05, 30E10, 46E15.
\end{AMS}

 \section*{Introduction}

In this paper, we extend the concept of Carleson measures to general
open subsets in the plane.  This work is the continuation of
\cite{ccarl} in which the author studies and characterizes all
Carleson measures on circular domains.

Let us begin with the Carleson measures on circular domains.  For a
domain $G$, let $R(\overline G)$ denote the uniform closure of
rational functions with poles off $\overline G$. A Carleson measure
on a circular domain is defined as:

\begin{definition}  \label {d:car}
A positive finite (Borel) measure on a circular domain $G$ is called
a Carleson measure if there exists a positive constant $C$ such that
for all $q\in [1,\infty)$
\[\{ \int_{G} |r|^{q} d\mu\}^{\frac{1}{q}} \leq C\hspace{.03in}
\{ \int_{\partial G} |r|^{q} d s\}^{\frac{1}{q}}, \hspace{.2in}
 r\in R(\overline G),
\]
where $s$ is the arclength measure on $\partial G$.
\end{definition}

Recall a Carleson square $C_{h}$ on the unit disk $D$  is a subset
of form
\[C_{h}= \{z=r e^{it}: 1-h \leq r < 1; \hspace{.04in}
t_{0} \leq t \leq t_{0}+h \}.
\]
A Carleson measure  was originally defined as a measure $\mu$ such
that $\mu(C_{h}) \leq c h$ for each Carleson square on $D$, where
$c$ is a positive constant. L. Carleson \cite{carla, carl}
established the equivalence between this definition and
Definition~\ref{d:car}.

Fortunately, it is also possible for us to define Carleson squares
on circular domains and it turns out that the definition is a
natural extension of Carleson squares on  the unit disk.

Let $G$ be a circular domain with boundary components $C_{0}$,
$C_{1}$, ... $C_{n}$ and let $C_{0}$ be the outer boundary. Let
$a_{i}$ be the center of $C_{i}$ and let $r_{i}$ be the radius of
the circle $C_{i}$ for $i=0,1,2, ... n$. Set
$f_{0}=\frac{z-a_{0}}{r_{0}}$ and $f_{i}=\frac{r_{i}}{z-a_{i}}$ for
each $i> 0$.
 We call a subset of $G$ a Carleson square if
$\overline S \cap \partial G$ is connected and if there is $ 0\leq i
\leq n$ such that $f_{i}(S)$ is a Carleson square on the unit disk
$D$.

The following two theorems are the main results in \cite{ccarl},
which characterize Carleson measures on circular domains and also
generalize a famous theorem of L. Carleson that establishes the
equivalence between the Carleson measure and the Carleson measure
inequality.

\begin{theorem} \label{t:0}
Let $G$ be a circular domain. A positive finite measure $\mu$ on $G$
is a Carleson measure if and only if there exists a positive
constant $c$ such that for each Carleson square $S \subset G$,
$\mu(S)  \leq c h,$ where $s$ is the arclength measure on $\partial
G$ and $h= s(\overline S \cap \partial G)$.
\end{theorem}


\begin{theorem} \label{t:1}
Let $G$ be a circular domain with boundary components $C_{0}$,
$C_{1}$, ... $C_{n}$ and let $C_{0}$ be the outer boundary.  For
each $i=0,1,2,...$, let $a_{i}$ be the center  and let $r_{i}$ be
the radius of  $C_{i}$, respectively. Set
$f_{0}=\frac{z-a_{0}}{r_{0}}$ and $f_{i}=\frac{r_{i}}{z-a_{i}}$ for
each $i> 0$. A positive measure on $G$ is a Carleson measure if and
only if $\mu\circ f_{i}^{-1}$ is a Carleson measure on the unit disk
$D$ for each $i$.
\end{theorem}

Originally motivated by a research problem in operator theory
suggested by Ronald Douglas: Can the similarity and quasi-similarity
of the well-known results of W. Clary  be extended to multiply
connected domains? The readers may consult
 \cite{conw}[ p. 370-373] and \cite{mcc,hast,equi,quas,nsim} for the topics.
The quasi-similarity case was first  investigated in \cite{mcc} and
then later in \cite{quas}, but the similarity case was only to be
done recently in \cite{nsim} and the extension was completed because
of the research in this work.  We consider the following problem:
What are the natural extension of the Carleson measures from
circular domains to general domains. Is there a meaningful and
satisfactory extension of Carleson measures to general domains or
even to general open subsets in the plane?

Carleson measures on the unit disk (or half plane) have many
applications and connections in analysis.  So we have reasons to
believe this problem is interesting and worth of studies in its own
right. Positive results toward the problem would likely have nice
applications in the function theory and operator theory on planar
domains (for an example, see \cite{nsim}).

Before we give the definition of Carleson measures, let us recall
what are harmonic measures for a given domain in the plane.

Let $G$ be a domain in the extended complex plane such that
$\partial G$ does not contain $\infty$. For $u \in C(\partial G)$,
let $\hat{u} =$ $\sup \{f: f$ is subharmonic on $G$ and
$\limsup_{z\rightarrow a} f(z) \leq u(a),\hspace{.02in} a \in
\partial G\}$. Then $\hat{u}$ is harmonic on $G$ and continuous on
$\overline{G}$, and the map $u\rightarrow \hat{u}(z)$ defines a
positive linear functional on $C(\partial G) $ with norm one, so the
Riesz representing theorem  implies that there is a probability
measure $\omega_{z}$ on $\partial G$ such that $  \hat{u}(z) = \int
_{\partial G} u d \omega_{z}, \hspace{.03in} u \in C(\partial G). $
The measure $\omega_{z}$ is called the harmonic measure for $G$
evaluated at $z$.

The harmonic measures evaluated at two different points are
boundedly equivalent . An observation is that a harmonic measure for
$G$ is boundedly equivalent to the arclength measure on $\partial G$
if the boundary $G$ consists of finitely many  smooth Jordan curves.

With the above observation, it is clear that we can  replace the
arclength by a harmonic measure for $G$  in Definition~\ref{d:car}.

Let $G$ be an open subset  in the complex plane and let $A(G)$
denote the function algebra consisting of all functions analytic on
$G$ and continuous on $\overline G$.  Now we define Carleson
measures on a general domain as follows:

\begin{definition}  \label {d:car2}
A positive finite (Borel) measure on a connected open subset $G$ is
called a Carleson measure if there exists a positive constant $C$
such that for all $q\in [1,\infty)$
\[\{ \int_{G} |r|^{q} d\mu\}^{\frac{1}{q}} \leq C\hspace{.03in}
\{ \int_{\partial G} |r|^{q} d \omega\}^{\frac{1}{q}}, \hspace{.2in}
 r\in A(G),
\]
where $\omega$ is a harmonic measure for $G$.
\end{definition}
Why do we use $A(G)$ instead of $R(\overline  G)$ in the definition
above? First, if $G$ is a circular domain, then there is no
difference. Second, when $G$ is a circular domain, we can actually
replace $R(\overline G)$ by $H^{\infty}(G)$, the algebra of bounded
analytic functions on $G$. However, it is impossible to use
$H^{\infty}(G)$ for a general domain $G$,  so $A(G)$ is likely the
largest algebra of analytic functions on $G$ that is most
appropriate for the definition.

Now the question is: Is this a natural extension? Our results show
that it seems the right extension for finitely connected domains.
Since the Carleson measures for infinitely connected domains are
uninteresting and have little meaning in general, our definition
appear to be a satisfactory one also.

In this work, we characterize all Carleson measures on multi-nicely
connected domains.  Moreover, our results, together with our
comments and conjectures, suggest  that this is the largest class of
finitely connected domains on which Carleson measures are
meaningful.

In Section 1, we study Carleson measures on multi-nicely connected
domains, and in Section 2 we extend the result in Section 1 to those
open subsets whose components are multi-nicely connected and the
harmonic measures of the components are mutually singular.
  \cite {dure}, \cite {ga} and \cite {conw} are good references for basic
concepts related to this work.

\section{Carleson Measures On Finitely Connected Domains} \label{s:b}
Let $G$ be a finitely connected domain whose boundary contains no
single point component.  It is a classical result that $G$ is
conformally equivalent to a circular domain in the complex plane.
Since we already know what are the Carleson measures on circular
domains, we wish to characterize Carleson measures on general
domains vis conformal maps.

The following proposition says that the conformal image of a
Carleson measure on a circular domain is also a Carleson measure.

\begin{proposition} \label {p:a}
Suppose $G$ be a domain conformally equivalent to a circular domain
$W$. If a positive finite measure $\mu$ on $G$ is mapped to a
Carleson measure on $W$ by a conformal map from $W$ onto $G$, then
$\mu$ is a Carleson measure.
\end{proposition}

\begin{proof}
Suppose $\mu=\tau\circ\phi$, where $\tau$ is a Carleson measure on
$W$ and $\phi$ is a conformal map from $G$ onto $W$. By our
definition, there is a positive constant $C$ such that for all $q\in
[1,\infty)$
\[   \|f\|_{L^{q}(\tau)}  \leq C\hspace{.03in}\|f\|_{L^{q}(\omega_{W})},
\hspace{.1in} f\in A(W),\] where $\omega_{W}$ is a harmonic measure
for $W$. Fix a number $q\in [1,\infty)$.
Let $\omega=\omega_{W}\circ \phi^{-1}$, then it is a harmonic
measure on $G$. Now, let $g\in A(G)$, then $g\circ\phi^{-1}\in
H^{\infty}(W)$. Since $W$ is a circular domain,   it is a classical
result that the boundary value function of $g\circ\phi^{-1}$   can
be approximated by a sequence $\{f_{n}\} \subset A(W)$ in
$L^{q}(\omega_{W})$ norm. If we still use $g\circ\phi^{-1}$ to
denote this boundary value function, then the above inequality
implies that
\[\lim_{n\rightarrow
\infty}\|g\circ\phi^{-1}-f_{n}\|_{L^{q}(\tau)}=0.\] So, we have
\[
 \| g\circ\phi^{-1}\|_{L^{q}(\tau)}=\lim_{n\rightarrow \infty} \| f_{n}\|_{L^{q}(\tau)}  \leq C
 \lim_{n\rightarrow \infty}\|f_{n}\|_{L^{q}(\omega_{W})}= \|
g\circ\phi^{-1}\|_{L^{q}(\omega_{W})}.   \]
Therefore,  it follows that  
 \begin{align*}
 \|g\|_{L^{q}(\mu)} = \|g\circ
 \phi^{-1}\|_{L^{q}(\tau)}
\leq C\hspace{.03in} \|g\circ\phi^{-1} \|_{L^{q}(\omega_{W})}   =
C\hspace{.03in}\|g\|_{L^{q}(\omega)}, \hspace{.04in}\mbox{for each}
\hspace{.04in}g\in A(G).
\end{align*}
Hence, $\mu$ is a Carleson measure on $G$.

\end{proof}
Later on, we will see: not every Carleson measure on $G$ is
generated in this way if $G$ is a general finitely connected domain.
However, as the following Theorem~\ref{t:a} shows that for the class
of multi-nicely connected domains, it is the case.

Let $G$ be a domain conformally equivalent to a circular domain $W$
and let $\alpha$ be a conformal map from $W$ onto $G$. By Fatou's
theorem, $\alpha$ has nontangential limits at almost every point of
$\partial W$ with respect to the arclength measure. So $\alpha$ has
a well-defined boundary value function up to a set of zero arclength
measure. We call $G$ is a multi-nicely connected domain if $\alpha$
is almost 1-1 on $\partial W$ with respect to the arclength measure
on $\partial W$. If $G$ is also simply connected, we call $G$ nicely
connected (a terminology used in \cite[Glicksburg]{glic}).

The following result is the main theorem in \cite{b_apr}:
\vspace{.1in}

\noindent{\bf Theorem A.}  {\em Let $G$ be a bounded open subset
such that each of its components is finitely connected and the
components of $G$ have no single point boundary components. Then in
order that for every function in $H^{\infty}(G)$ be the pointwise
limit of a bounded sequence in $A(G)$ it is necessary and sufficient
that the harmonic measure of the components of $G$ are mutually
singular and each component of $G$ is multi-nicely connected. }

\vspace{.1in} The following theorem gives a characterization of
Carleson measures on multi-nicely connected domains.
\begin{theorem} \label{t:a}
Let $G$ be a multi-nicely connected domain conformally equivalent to
a circular domain $W$.  Let $\alpha$ denote a conformal map of $W$
onto $G$. Then a positive finite measure $\mu$ on $G$ is a Carleson
measure if and only if $\mu\circ\alpha$ is a Carleson measure on
$W$.
\end{theorem}

\begin{proof}
Necessity.

Since $G$ is conformally equivalent to $W$ and since $G$ is
multi-nicely connected, there exists a conformal map $\psi$ from $W$
onto $G$ such that the boundary value function of $\psi$  is
univalent almost everywhere on $\partial W$.  We also use $\psi$ to
denote its boundary value function. So there is a measurable subset
$E$ with full harmonic measure so that $\psi$ is univalent on $E$.
Thus, $\psi^{-1}$ is well-defined on $\psi(E)$, which has full
harmonic measure. Let $\omega$ denote a harmonic measure for $G$.
Then $\psi^{-1}$ is a well-defined function in $L^{q}(\omega)$ for
all $1 \leq q \leq \infty$.  Suppose that $\mu$ is a Carleson
measure on $G$. Then there is a positive constant $C$ such that for
all $q\in [1,\infty]$
\[\{ \int_{G} |f|^{q}d\mu\}^{\frac{1}{q}} \leq C\hspace{.03in} \hspace{.02in}
\{\int_{\partial G} |f|^{q} d\omega \}^{\frac{1}{q}}, \hspace{.2in}
r\in A(G).  \] Pick a function $g$ in $A(W)$.  Then $g\circ\psi^{-1}
\in H^{\infty}(G)$.  Since $G$ is a multi-nicely connected, it
follows (Theorem A) that $g\circ\psi^{-1}$ is the limit of a bounded
sequence of functions in $A(G)$. This, in turn, implies that $g$ is
the  weak-star limit of a sequence of functions of $A(G)$ in
$L^{\infty}(\omega)$. Let  $A^{\infty}(G,\omega) $ and
$A^{q}(G,\omega)$    denote the weak-star closure  of $A(G)$ in
$L^{\infty}(\omega)$ and the mean closure of $A(G)$ in
$L^{q}(\omega)$, respectively.  We now  show that $A^{\infty}(G,
\omega) \subset A^{q}(G,\omega)$ for all $q\in [1, \infty)$. In
fact, Suppose that  the inclusion does not hold. Then using
Hahn-Banach theorem, we can find $x\in A^{\infty}(G, \omega)$ and
$h\in A^{q}(G,\omega)^{*}= A^{p}(G,\omega)$ (where $\frac{1}{q}
+\frac{1}{p}=1)$
 such that
\[ \int xh \hspace{.03in}d\omega=1 \hspace{.3in}\mbox{and}\hspace{.3in} \int
fh \hspace{.03in}d\omega=0\hspace{.3in}\mbox{for each}\hspace{.3in}
f\in A^{q}(G,\omega).\] Let $\{g_{n}\}$ be a sequence such that it
weak-star converges to $x$. Then by the definition of weak-star
topology, we have
\[ \int  x f   \hspace{.03in}d\omega= \lim_{n\rightarrow \infty} \int g_{n} f  \hspace{.03in}d\omega\hspace{.3in}\mbox{for
each}\hspace{.3in} f\in L^{1}(\omega).\]
 In particular, the above
equality holds for $h$ and so we get that $\int xh
\hspace{.03in}d\omega=0,$  which is a contradiction.

 Now fix such a number $q$. For any function
$g$ in $A(W)$, choose a sequence $\{f_{n}\}$ in $A(G)$ such that
$f_{n}\rightarrow g\circ\psi^{-1}$ in $A^{q}(G,\omega)$. Then we
have
\begin{align*}
\{ \int_{W} |g|^{q} d\mu\circ\psi\}^{\frac{1}{q}}& = \{ \int_{G}
|g\circ \psi^{-1}|^{q} d\mu\}^{\frac{1}{q}} = \lim_{n\rightarrow
\infty} \{\int_{G} |f_{n}|^{q} d\mu \}^{\frac{1}{q}}
\\&
\leq C\hspace{.015in}\lim_{n\rightarrow \infty} \{\int_{G}
|f_{n}|^{q} d\omega
 \}^{\frac{1}{q}}
 = C\hspace{.02in} \{ \int_{G} |g\circ \psi^{-1}|^{q}
d\omega\}^{\frac{1}{q}} \\& = C\hspace{.02in} \{ \int_{\partial W}
|g|^{q} d\omega\circ\psi\}^{\frac{1}{q}}.
\end{align*}
Since $\omega\circ\psi$ is boundedly equivalent to the arclength
measure on $\partial W$, we conclude that $\mu\circ\psi$ is a
Carleson measure on $W$. Hence, it follows by Proposition~\ref{p:a}
that $\mu\circ \alpha= \mu\circ \psi\circ(\psi^{-1}\circ\alpha)$ is
a Carleson measure on $W$.

Sufficiency. This directly follows from Proposition~\ref{p:a}.

\end{proof}

\begin{corollary} \label{c:invi}
Suppose  $G$ and $V$ are multi-nicely connected domains. Let
$\alpha$ be a conformal map from $G$ onto $V$. Then $\alpha^{-1}$
maps a Carleson measure on $G$ to a Carleson measure on $V$.
\end{corollary}

\begin{proof}
Let $\mu$ be a Carleson measure on $G$. Since $G$ is multi-nicely
connected, there is a circular domain $W$ and a comformal map $\psi$
from $W$ onto $G$. By the above theorem, $\mu\circ\psi$ is a
Carleson measure on $W$. Since $\psi^{-1}\circ\alpha^{-1}$ is a
conformal map from $V$ onto $W$, it follows by Theorem~\ref{t:a}
again that
$\mu\circ\alpha^{-1}=\mu\circ\psi\circ(\psi^{-1}\circ\alpha^{-1})$
is a Carleson measure on $V$.

\end{proof}

\begin{corollary} \label {c:indep}
Let $G$ be a multi-nicely connected  domain.  A positive finite
measure $\mu$ on $G$ is a Carleson measure if and only if there is a
positive constant $C$ such that for some $p\in [1,\infty)$ $
\|f\|_{L^{p}(\mu)} \leq C \|f\|_{L^{p}(\omega)},\hspace{.03in}f\in
A(G). $
\end{corollary}

\begin{proof}
It is well-known that the conclusion of the corollary is true if $G$
is the unit disk. Combining this fact with the proof of
Theorem~\ref{t:1}, we see that Corollary is also true when $G$ is a
circular domain. Therefore, it follows by Theorem~\ref{t:a} that the
conclusion holds for multi-nicely connected domains.

\end{proof}

Now an interesting question is raised:
\begin{quote}
For what domain $G$ on which every Carleson measure $\mu$ is of the
form $\tau\circ\phi$, where $\tau$ is a Carleson measure on $W$ and
$\phi$ is a conformal map from a circular domain $G$ onto $W$?
\end{quote}

We guess that the answer to the question is: $G$ is a multi-nicely
connected domain. We are unable to prove this conjecture though we
strongly believe it is the case. Nevertheless, we can prove the
following: \vspace{.1in}

{\em Suppose $G$ is a domain such that every Carleson measure $\mu$
on $G$ is of the form $\tau\circ\phi$. Then no point of $\partial G$
can have such a neighborhood $V$ for which $V\setminus \partial G
\subset G$ and $\partial G\cap V$ is smooth Jordan arc. }
\vspace{.1in}

We outline a proof of this fact in the following:

\vspace{.1in}


Suppose such a domain $V$ exists. We first assume that $V$ is a
Jordan domain with smooth boundary.  An application of Morera's
theorem (see \cite{conw0}) shows that $A(G)= A(G\cup V)$. Moreover,
we next show that $V$ is contained in the set of analytic bounded
point evaluations for $A^{q}(G,\omega)$ for all $q\in [1,\infty)$,
where $A^{q}(G,\omega)$ is the closure of $A(G)$ in $L^{q}(\omega)$.

Recall a point $w$ in the complex plane is called an ($abpe$)
analytic bounded point evaluation for $A^{q}(G,\omega)$ if there is
a positive number $c$ and a neighborhood $V_{w}$ of $w$ such that
for every $z\in V_{w}$ $ |f(z)| \leq c
\|f\|_{L^{q}(\omega)},\hspace{.02in}f\in A(G).$ We use $\nabla
A^{q}(G,\omega)$ to denote the set of $abpe$s (consult \cite{conw}
for $abpe$s).

Let $\omega_{1}$ denote the sum of the  harmonic measures of the
(two) components of $V\setminus \partial G$. Then $\omega_{1}$ is
supported on $\partial (V \setminus \partial G)$. By Harnack's
inequality, we see that $V\setminus \partial G$ is contained in
$\nabla A^{q}(V,\omega_{1})$. Since the restrictions of $\omega_{1}$
to the boundaries of the components of $V\setminus \partial G$ are
not mutually singular, we see that $\nabla A^{q}(V,\omega_{1}) \neq
V\setminus \partial G. $ So $\nabla A^{q}(V,\omega_{1})$ is a simply
connected domain such that $V\setminus \partial G \subset  \nabla
A^{q}(V,\omega_{1}) \subset V. $ According to Lemma 3 in \cite{comm}
(or see the main theorem in \cite{appr}), $\nabla
A^{q}(V,\omega_{1})$ must be nicely connected. Since $V\cap \partial
G$ is a smooth Jordan arc, it follows that $\nabla
A^{q}(V,\omega_{1}) = V$. Using the maximum principle, it can be
shown that $\nabla A^{q}(V,\omega_{1})\subset\nabla
A^{q}(G,\omega).$ Hence, we have that $V\subset \nabla
A^{q}(G,\omega).$ This proves the claim.

Evidently, we can use a Jordan domain with smooth boundary to
replace $V$ if it is not one. Moreover, we certainly can require
that $\overline V \subset \nabla A^{q}(G,\omega)$. Therefore, it
follows from the definition of analytic bounded point evaluations
that there is a positive constant $C$ such that  for each $z\in
\overline V$
\[ (*)\hspace{.2in} |f(z)| \leq c \|f\|_{L^{q}(\omega)},\hspace{.2in}f\in A(G).\]
Let $\alpha$ be a conformal map from $W$ onto $G$ and let $S$ be a
Carleson square contained in $W$ such that $\alpha(S) \subset V$.
Choose a positive non-Carleson measure $\nu$ on $W$ so that the
support of $\nu$ is contained in $\overline S$ (the existence of
such a measure is no problem).  Set $\mu=\nu\circ\alpha^{-1}$. Then,
it follows from $(*)$ that $ \|f\|_{L^{q}(\mu)} \leq c_{1}
\|f\|_{L^{q}(\omega)}, \hspace{.2in}f\in A(G),$ where $c_{1}$ is a
positive constant. Thus, by definition $\mu$ is a Carleson measure
on $G$.  Now, our assumption implies that $\mu=\nu\circ\alpha^{-1}=
\tau\circ\phi$, where $\tau$ is a Carleson measure on $W$ and $\phi$
is a conformal map of $G$ onto $W$. Consequently, we have that
$\nu=\tau\circ\phi\circ\alpha$, which implies that $\nu$ is a
Carleson measure on $W$. This contradicts our choice of $\nu$ and
hence the proof is complete.

\begin{conjecture} \label {cj:1}
Let $G$ be a domain conformally equivalent to a circular domain $W$
and let $\phi$ be a conformal map of $G$ onto $W$. If every Carleson
measure on $G$ is of the form  $\tau\circ\phi$, where $\tau$ is a
Carleson measure on $W$, then $G$ is a multi-nicely domain.
\end{conjecture}

A positive answer to the conjecture will show that Carleson measures
on non-multi-nicely connected domains are not interesting and have
little meaning. So the extend of Carleson measures would be up to
the class of multi-nicely connected domains.

Recall the connectivity of a connected open subset is defined to be
the number of the components of the complement of $G$. The Riemann
mapping theorem insures that all simply connected  domains (i.e.
those domains having the connectivity 1) are conformally equivalent.
However, the same is not true for the domains having higher
connectivity. Because of this, the characterization of Carleson
measures in Theorem~\ref{t:a}  does not give us complete
satisfaction.

Next, we give another characterization of Carleson measures in
Theorem~\ref{t:main}, which, together with Theorem~\ref{t:a},
completes the picture of Carleson measures on  multi-nicely
connected domains.

We first prove the following lemma.

\begin{lemma}  \label {l:equi}
Let $G$ be a multi-nicely connected domain and let $\mu$ be a
positive finite measure on $G$.  Let $E$ be a compact subset of $G$.
Define $\mu_{E}= \mu| (G\setminus E)$.  If $\mu_{E}$ is a Carleson
measure on $G$, then $\mu$ is a Carleson measure on $G$.
\end{lemma}
\begin{proof}
We first assume that $G$ is a circular domain. Suppose that
$\mu_{E}$ is a Carleson measure $G$.  By definition, there exists a
positive constant $C$ such  that for each $q\in [1,\infty)$
\[ \| f\|_{L^{q}(\mu_{E})} \leq C\hspace{.03in}
\|f\|_{L^{q}(\omega)}, \hspace{.2in} f \in A(G),
\]
where $\omega$ is a harmonic measure of $G$. Fix such  a number $q$.
Let $\{f_{n}\}$ be a sequence of functions in $A(G)$ such that it
converges to zero in $L^{q}(\omega)$. Since $G$ is a circular
domain, it is a classical result that $\{f_{n}\}$ converges to zero
uniformly on compact subsets of $G$. This together with the
inequality above implies that $\lim_{ n\rightarrow \infty} f_{n} =
0\hspace{.03in}\mbox{ in } \hspace{.03in}L^{q}(\mu). $ Let
$A^{q}(G,\mu)$ ($A^{q}(G,\omega)$) be the closure of $A(G)$ in
$L^{q}(\mu)$ ($L^{q}(\omega)$). Then, it follows that the linear
operator $A$, from $A^{q}(G,\mu)$ to $A^{q}(G,\omega)$, densely
defined by $ A(f) = f$ for each $f\in A(G)$ is continuous.
Therefore, we have
\[ \| f\|_{L^{q}(\mu)} \leq \|A\|\hspace{.03in}
\|f\|_{L^{q}(\omega)}, \hspace{.2in} f\in  A(G).
\]
Applying Corollary~\ref{c:indep},  we conclude that $\mu$ is a
Carleson measure on $G$.

Finally,  using Theorem~\ref{t:a}, we conclude that the lemma holds
for an arbitrary multi-nicely connected domain.

\end{proof}

The next theorem gives another characterization of Carleson measures
on nicely connected domains, which also generalizes
Theorem~\ref{t:1}.

\begin{theorem}  \label {t:main}
Suppose that  $G$ is a multi-nicely connected domain. Let
$\{E_{n}\}_{n=0}^{l}$ be the collection of all components of
$\partial G$, and let $\widehat{G}_{n}$ be the simply connected
domain containing $G$ and having $E_{n}$ as its boundary in the
extended plane. Let $f_{n}$ denote a Riemann map of
$\widehat{G}_{n}$ onto $D$. Then a positive finite measure $\mu$ on
$G$ is a Carleson measure if and only if $\mu\circ f^{-1}_{n}$ is a
Carleson measure on $D$ for each $0\leq n \leq l$.
\end{theorem}

\begin{proof}
Sufficiency.

Suppose that $\mu\circ f^{-1}_{n}$ is a Carleson measure on $D$ for
each $0\leq n\leq  l$.  Fix an integer $n\leq l$ and let $\alpha$ be
a conformal
map of $\widehat G_{n}$ onto $D$. Then
$\alpha(G) \subset D$.  By our hypothesis, $\widehat G_{n}$ is a
nicely connected domain.

Let $\epsilon $ is a sufficiently small positive number so that
\[R_{2\epsilon} = \{z: 1-2\epsilon < |z| < 1 \}\subset \alpha(G) \]
and define   a measure  $ \nu_{n}$ by let $\nu_{n}$  be the
restriction of $\mu\circ \alpha^{-1} $ on  $R_{\epsilon}$.  Then,
our hypothesis implies that $\nu_{n}$ is a Carleson measure on $D$.
An application of Theorem~\ref{t:1} shows that $\nu_{n}$ is a
Carleson measure on $R_{2\epsilon}$.  Since $R_{2\epsilon} \subset
\alpha(G)$,  we can choose a harmonic measure
$\omega_{R_{2\epsilon}}$ for $R_{2\epsilon}$ and a harmonic measure
$\omega_{\alpha(G)}$ for $\alpha(G)$ so that they have the same
evaluating point. Then, there exists a positive constant $c$ such
that for each
$q\in [1,\infty)$ 
\[ \|g\|_{L^{q}(\nu_{n})} \leq c\hspace{.03in} \|g\|_{L^{q}(\omega_{R_{2\epsilon}})}
\leq c\hspace{.03in} \|g\|_{L^{q}(\omega_{\alpha(G)})},\hspace{.3in}
g\in A(\alpha(G)).
\]
Since $G$ is multi-nicely connected and $\alpha$ is a conformal map
onto a circular domain $D$,  it is  not difficult to see that
$\alpha(G)$ is also a multi-nicely connected domain. So
$A(\alpha(G))$ is pointwise boundedly dense in
$H^{\infty}(\alpha(G))$.  Therefore, we have that  for each $q\in
[1,\infty)$
\begin{align*}
 \|f\|_{L^{q}(\nu_{n}\circ\alpha) } = \|f\circ\alpha^{-1}\|_{L^{q}(\nu_n)}
\leq c\hspace{.03in}
\|f\circ\alpha^{-1}\|_{L^{q}(\omega_{\alpha(G)})} \leq
c\hspace{.03in} \|f\|_{L^{q}(\omega)},\hspace{.05in} f\in A(G),
\end{align*}
where $\omega$ is a harmonic measure for $G$. By definition,
$\nu_{n}\circ\alpha$ is a Carleson measure on $G$. Therefore, if we
set $\nu= \sum_{n}^{l} \nu_{n=1}$, then $\nu$ is a Carleson on
measure $G$.

Lastly, applying Lemma~\ref{l:equi}, we conclude that $\mu$ is a
Carleson measure on $G$.

Necessity.

Suppose that $\mu$ is a Carleson measure on $G$. Fix an integer
$n\geq 1$.
 Let $W$ be a circular domain conformally equivalent to $G$. Let $\lambda$
be a conformal map from $\widehat G_{n}$ onto $\widehat W$, which is
the disk having the outer boundary of $W$ as its boundary. According
to Corollary~\ref{c:invi},  $\mu\circ\lambda^{-1}$ is a Carleson
measure on $\lambda(G)$. Now, let $\tau$ denote the Riemann map from
$\widehat W$ onto $D$ that sends the center of $\widehat W$ to the
origin and has positive derivertive at the center point of $\widehat
W$. Then it follows by Theorem~\ref{t:1} that $
\mu\circ\lambda^{-1}\circ\tau^{-1}$ is a Carleson measure $D$. Since
$\mu\circ f_{n}^{-1} = \mu\circ\lambda^{-1}\circ\tau^{-1} \circ
(\tau\circ\lambda\circ f_{n}^{-1}) $ and since
$(\tau\circ\lambda\circ f_{n}^{-1})$ is a conformal map from $D$ to
$D$, it follows by Theorem~\ref{t:a}  that $\mu\circ f_{n}$ is a
Carleosn measure on $D$. So we are done.

\end{proof}

\section{Carleson measures on open subsets}

Let $G$ be an open subset with components $G_{1}, G_{2}, ... G_{n},
...$ on the plane. Recall that a harmonic measure for $G$ is defined
to be $\omega = \sum 2^{-n} \omega_{G_{n}}$, where $\omega_{G_{n}}$
is a harmonic measure for $G_{n}$. Observe that $\omega$ is a
measure supported on $\partial G$.

Note, if $G$ has only finitely many components, then two harmonic
measures for $G$ are still boundedly equivalent. But this is no
longer true when $G$ has infinitely many components.

Now we define Carleson measures on a general open subset in the
following:

\begin{definition}  \label {d:car3}
Let $G$ be an open subset and let $\omega$ be a harmonic measure for
$G$.  A positive finite measure on $G$ is said to be  a Carleson
measure (with respect to $\omega$) if there exists a positive
constant $C$ such that for all $q\in [1,\infty)$
\[\{ \int_{G} |r|^{q} d\mu\}^{\frac{1}{q}} \leq C\hspace{.03in}
\{ \int_{\partial G} |r|^{q} d \omega\}^{\frac{1}{q}}, \hspace{.2in}
 r\in A(G).
\]
\end{definition}

The definition of Carleson measures is still independent to the
choise of the harmonic measure $\omega$ when $G$ has only finitely
many components.

Conjecture~\ref{cj:1} suggest that Carleson measures on
non-multi-nicely connected domains are uninteresting. So in this
section, we shall only consider those open subsets whose components
are {\em all} multi-nicely connected domains.

In the following lemma, $G$ is an open subset with components
$\{G_{n}\}_{n\geq 1}$ and $\omega$ denotes a harmonic measure for
$G$.

\begin{lemma} \label{l:b}
Suppose that  each  $G_{n}$ is a multi-nicely domain and
 the harmonic measures of the components of $G$ are mutually singular.
 Then a positive finite measure $\mu$ on $G$ is a Carleson
measure if and only if there is a positive constant $C$ such that
for each $q\in [1,\infty)$ and for each $n \geq 1$
\[ \{\int_{G_{n}} |r|^{q} d\mu\}^{\frac{1}{q}} \leq C \hspace{.02in}
\{\int_{F_{n}} |r|^{q} d\omega\}^{\frac{1}{q}}, \hspace{.2in} r\in
A(G_{n}),
\]
where $F_{n}$ is  a carrier of a harmonic measure of $G_{n}$ for
each $n$.
\end{lemma}
\begin{proof}
Necessary.

Suppose that $\mu$ is a Carleson measure on $G$.  Then
 there is a positive constant $C$ such that for each $q\in [1,\infty)$
\[ \{\int |r|^{q} d\mu\}^{\frac{1}{q}} \leq C \hspace{.02in}
\{\int |r|^{q} d\omega\}^{\frac{1}{q}} , \hspace{.2in} r\in A(G).
\]
By Theorem A, we see that our hypothesis implies that each function
in $H^{\infty}(G)$ is the pointwise limit of a bounded sequence of
functions in $A(G).$

Fix an integer $n$ and pick a function $r\in A(G_{n})$. Let
$\chi_{G_{n}\cup F_{n}}$ denote the characteristic function of
$G_{n}\cup F_{n}$. Then there is a bounded sequence in $A(G)$ such
that it pointwise converges to $r \chi_{G_{n}\cup F_{n}}$ on
$G_{n}$. But this, together with our hypothesis, implies that
$r\chi_{G_{n}\cup F_{n}}$ (actually, the boundary value function
$\chi_{G_{n}\cup F_{n}}$ on $\partial G$) is in the weak-star
closure of $A(G)$ in $L^{\infty}(\omega)$. Thus, it follows that
$r\chi_{G_{n}\cup F_{n}}$  belongs to the mean closure of $A(G)$ in
$L^{q}(\omega)$ for all $q\in [1,\infty)$.  Fix a number $q\in
[1,\infty)$ and choose a sequence of functions $\{f_{k}\}\subset
A(G)$ such that $f_{k} \rightarrow r\chi_{G_{n}\cup F_{n}}$ in $
L^{q}(\omega)$. Then we have that  $f_{k}\rightarrow
r\chi_{G_{n}\cup F_{n}}$ in $L^{q}(\mu)$ as well.  Therefore, we
conclude that for each $r\in A(G_{n})$
\begin{align*}
\{\int_{G_{n}} |r|^{q} d\mu\}^{\frac{1}{q}} & \leq
 \{\int |r \chi_{G_{n}\cup F_{n}}|^{q} d\mu\}^{\frac{1}{q}}  =
\lim_{k\rightarrow\infty}\{\int_{G}|f_{n}|^{q}d\mu\}^{\frac{1}{q}}\\&
\leq C \lim_{k\rightarrow\infty}\{\int_{\partial
G}|f_{n}|^{q}d\omega\}^{\frac{1}{q}} =
C\hspace{.02in}\{\int|r\chi_{G_{n}\cup
F_{n}}|^{q}d\omega\}^{\frac{1}{q}}\\& = C
\hspace{.02in}\{\int_{F_{n}} |r|^{q} d\omega\}^{\frac{1}{q}}.
\end{align*}

Sufficiency. Observe that $\mu(G)=\sum \mu(G_{n})$ and
$\omega(\partial G) =\sum \omega(F_{n})$. So it follows evidently
that the conclusion is true.

\end{proof}


Recall that we define $\omega =\sum_{n=1}
\frac{1}{2^{n}}\omega_{n}$, where $\omega_{n}$ is a harmonic measure
for $G_{n}$. The condition in Lemma~\ref{l:b} is made on $\omega$.
If we replace $\omega|F_{n}$ by $\omega_{n}$ in Lemma~\ref{l:b},
then we have our main theorem of the section:

\begin{theorem} \label {t:c}
Suppose that each $G_{n}$ is multi-nicely connected and the harmonic
measures of the components of $G$ are mutually singular. Then a
positive finite measure $\mu$ on $G$ is a Carleson measure if and
only if each $\mu|G_{n}$ is a Carleson measure and there is a
sequence of positive numbers $\{c_{n}\}_{n=1}$ such that $\{c_{n}\}
\sim \{\frac{1}{2^{n}}\}$
 o($\frac{1}{n}$) and for each $n\geq 1$
\[\|r\|_{L^{q}(\mu|G_{n})} \leq c_{n} \|r\|_{L^{q}(\omega_{n})},\hspace{.05in}
\hspace{.05in} r\in A(G_{n}).
\]
\end{theorem}
\begin{proof}
Necessity.

Suppose that $\mu$ is a Carleson measure on $G$. Let $F_{n}$ denote
a carrier of $\omega_{n}$. According to Lemma~\ref{l:b}, there is a
positive constant $C$ such that  for each  $n\geq 1$ and for each
$q\in [1,\infty)$
\[\|r\|_{L^{q}(\mu|G_{n})} \leq C\hspace{.02in}
\|r\|_{L^{q}(\omega|F_{n})}, \hspace{.1in} r\in A(G_{n}).
\]
By our hypothesis, we have $\omega_{n} \perp \omega_{m}$ (if $n\neq
m$). So the definition of $\omega$ implies that $\omega|F_{n} =
\frac{\omega_{n}}{2^{n}}$. It follows that for each $n\geq 1$ and
each $q\in [1,\infty)$
\[\|r\|_{L^{q}(\mu|G_{n})} \leq \frac{C}{2^{\frac{n}{q}}}
 \|r\|_{L^{q}(\omega_{n})},\hspace{.05in}
\hspace{.05in} r\in A(G_{n}).
\]
In particular, this says that  $\mu|G_{n}$ is a Carleson measure on
$G_{n}$ for each $n$. Moreover, when $q=1$, we have that for each
$n\geq 1$
\[(*)\hspace{.5in} \|r\|_{L^{1}(\mu|G_{n})} \leq \frac{C}{2^{n}}
 \|r\|_{L^{1}(\omega_{n})},\hspace{.05in}
\hspace{.05in} r\in A(G_{n}).
\]
Since $\mu|G_{n}$ is a Carleson measure and $G_{n}$ multi-nicely
connected, it follows by Corollary~\ref{c:indep}  that
for each $n\geq 1$ and for all $q\in [1,\infty)$
\[\|r\|_{L^{q}(\mu|G_{n})} \leq \frac{C}{2^{n}}
 \|r\|_{L^{q}(\omega_{n})},\hspace{.05in} \hspace{.05in} r\in A(G_{n}).
\]
Now, if we let $c_{n}= \frac{C}{2^{n}}$ for each $n \geq 1$, then
$\{c_{n}\} \sim \{ \frac{1}{2^{n}}\}$ $o(\frac{1}{n})$, and moreover
\[\|r\|_{L^{q}(\mu|G_{n})} \leq c_{n} \|r\|_{L^{q}(\omega_{n})},\hspace{.05in}
\hspace{.05in} r\in A(G_{n}).
\]

Sufficiency.

Suppose that there is a sequence of positive numbers
 $\{c_{n}\}_{n=1}$ such that
\[\{c_{n}\} \sim \{\frac{1}{2^{n}}\}\hspace{.04in} o(\frac{1}{n})\hspace{.04in}
\mbox{ and }\hspace{.02in}
 \|r\|_{L^{q}(\mu|G_{n})} \leq c_{n} \|r\|_{L^{q}(\omega_{n})},\hspace{.05in}
\hspace{.03in} r\in A(G_{n}).
\]
Then there is a sufficiently large integer $N$ such that $c_{n} \leq
\frac{C^{'}}{2^{n}} $ when $ n\geq N $,  where $ C^{'}$ where is a
positive constant.

Let $C^{''} = \max_{i\leq N-1} \{C^{'}$, $c_{i}\}$ and set $ C=
2^{N}C^{''}$. Then $c_{n} \leq \frac{C}{2^{n}}$ for each $n\geq 1$.
Thus, we have that for each $q\in [1,\infty)$
\begin{align*}
 \leq C(\frac{1}{2^{n}})^{\frac{1}{q}} \|r\|_{L^{q}(\omega_{n})}
 = C\|r\|_{L^{q}(\omega|G_{n})} \hspace{.03in}
\mbox{ for }\hspace{.03in} r\in A(G_{n}).
\end{align*}
Consequently, for each $r\in A(G)$ we get that
\begin{align*}
\|r\|_{L^{q}(\mu)} = \{\sum_{n=1} \int |r|^{q}
d\mu|G_{n}\}^{\frac{1}{q}}    \leq \{C^{q}\sum_{n=1} \int |r|^{q}
d\omega|G_{n}\}^{\frac{1}{q}}
= C \|r\|_{L^{q}(\omega)}.
\end{align*}
Therefore, $\mu$ is a Carleson measure on $G$.

\end{proof}

\begin{remark}
When $n=\infty $, the definition of Carleson measures we employ as
in Definition~\ref{d:car3} depends on the choice of the harmonic
measure $\omega$. In fact, we define a harmonic measure as:
$\omega=\sum \frac{1}{2^{n}} \omega_{n}$.  Clearly, we can chose
another sequence of positive numbers $\{b_{n}\}$ so that $\omega_{b}
=\sum  b_{n}\omega_{n}$ is also a probability measure and
$\omega_{b}$ clearly can also be called a harmonic measure. Though
we still have that  $[\omega_{b}]=[\omega]$, $\omega_{b}$ and
$\omega$ are no longer boundedly equivalent in general.  So a
measure is a Carleson measure (with respect to $\omega_{b}$) if and
only if the restriction of $\mu$ to each of the components of $G$ is
a Carleson measure and there is a sequence of positive numbers
$\{c_{n}\}_{n=1}^{\infty}$ such that $\{c_{n}\} \sim \{b_{n}\}$ as
$n\rightarrow \infty$.

\end{remark}

Carleson measures in Theorem~\ref{t:c}  have a nice application to
operator theory. In \cite{nsim} the author characterizes a class of
operators similar to the operator of multiplication by $z$ on the
Hardy space $H^{2}(G)$ (where $G$ has infinitely many components)
via Carleson measures on $G$.

For an open subset with finitely many components, we have:

\begin{corollary}
Let $G=\sum_{n=1}^{k} G_{n}$. Suppose each $G_{n}$ is multi-nicely
connected and the harmonic measures of the components of $G$ are
mutually singular. Then a positive measure $\mu$ on $G$ is a
Carleson measure if and only if the restriction to $G_{n}$ is a
Carleson measure for each $n$.
\end{corollary}proposition

We finish the paper by post the following conjecture closely related
to Conjecture~\ref{cj:1}.

\begin{conjecture}
Suppose $G$ be an open subset such that if $\mu$ is a Carleson
measure, then the restrictions of $\mu$ to each of the components of
$G$ are Carleson measures. Then the harmonic measures of the
components of $G$ are mutually singular.
\end{conjecture}

Using the same proof as  in the comments before
conjecture~\ref{cj:1}, we can show that if $G$ satisfies the
condition of the above conjecture, there is no point in $\partial G$
can have a neighborhood $V$ for which $V\setminus
\partial G \subset G$ and $\partial G\cap V$ is a smooth Jordan arc.

\end{document}